\documentclass[12pt, reqno]{amsart}
\usepackage{amsmath, amssymb, latexsym}
\usepackage{graphicx}

\usepackage{comment}

\title{Computing Severi Degrees with Long-edge Graphs}

\author[F.\ Block]{Florian Block}
\address{Department of Mathematics,
University of California, Berkeley, Berkeley, CA 94720, USA}
\email{block@math.berkeley.edu}

\author[S.\ J.\ Colley]{Susan Jane Colley}
\address{Department of Mathematics,
Oberlin College, Oberlin, Ohio 44074, USA}
\email{sjcolley@math.oberlin.edu}

\author[G.\ Kennedy]{Gary Kennedy}
\address{Ohio State University at Mansfield, 1680 University Drive,
Mansfield, Ohio 44906, USA}
\email{kennedy@math.ohio-state.edu}


\newtheorem{thm}[equation]{Theorem}
\newtheorem{lem}[equation]{Lemma}
\newtheorem{prop}[equation]{Proposition}

\theoremstyle{definition}
\newtheorem{defn}[equation]{Definition}
\newtheorem{example}[equation]{Example}

\theoremstyle{remark}
\newtheorem{rem}[equation]{Remark}

\numberwithin{equation}{section}

\newcommand{\C}{\mathbb{C}}

\newcommand{\PP}{\mathbb{P}}

\newcommand{\SN}{\mathcal{N}}
\newcommand{\SNb}{\overline{\mathcal{N}}}
\newcommand{\SNbb}{\overline{\overline{\mathcal{N}}}}
\newcommand{\SQ}{\mathcal{Q}}
\newcommand{\SQb}{\overline{\mathcal{Q}}}
\newcommand{\SQbb}{\overline{\overline{\mathcal{Q}}}}
\newcommand{\SD}{\mathcal{D}}
\newcommand{\SP}{\mathcal{P}}
\newcommand{\calS}{\mathcal{S}}

\DeclareMathOperator{\ext}{ext}
\DeclareMathOperator{\Edge}{Edge}

\begin{document}

\begin{abstract}
We study a class of graphs with finitely many edges in order to understand the nature of the formal logarithm of the generating series for Severi degrees in elementary combinatorial terms.  These graphs are related to floor diagrams associated to plane tropical curves originally developed in \cite{BM} and used in \cite{Block} and \cite{FM} to calculate Severi degrees of $\PP^2$ and node polynomials of plane curves.
\end{abstract}

\maketitle

%
%
%
%
%

\section{Introduction} \label{intro}

\par
The motivating question for this article is classical and well-known, namely to determine the number $N^{d,\delta}$ of (possibly reducible) curves in $\PP^2_{\C}$ of degree $d$ having $\delta$ nodes and passing through $\dfrac{d(d+3)}{2} - \delta$ general points.   This number $N^{d,\delta}$ is the degree of the Severi variety.  When $d \ge \delta + 2$, the curves in question are irreducible, so that $N^{d,\delta}$ coincides with the Gromov--Witten invariant $N_{d,g}$, where $g = \dfrac{(d-1)(d-2)}{2} - \delta$.     

\par
Despite its long history, there continues to be interest in the Severi degree and much recent activity surrounding it.  In \cite{DiFI} Di~Francesco and Itzykson conjectured that $N^{d,\delta}$ is given by a \emph{node polynomial} $N_\delta(d)$ for sufficiently large $d$ and fixed $\delta$.  The polynomiality of $N^{d,\delta}$ became part of G\"ottsche's larger conjecture \cite[Conjecture 4.1]{G}---recently established by Tzeng in \cite{Tz} and independently by Kool, Shende, Thomas \cite{KST}---regarding the existence of universal polynomials enumerating curves on smooth projective surfaces.  The so-called threshold of polynomiality, i.e., the value $d^*$ such that the Severi degree $N^{d,\delta}$ is given by a polynomial for all $d \ge d^*$, has been steadily lowered.  In the proof of Theorem 5.1 of \cite{FM}, Fomin and Mikhalkin showed that $d^* \le 2\delta$; this was improved to $d^* \le \delta$ by the first author in \cite{Block}.  In the past year the bound for $d^*$ was sharpened still further to at most $\lceil \delta/2 \rceil +1$ (for $\delta \ge 3$) by Kleiman and Shende in \cite{KS}; this result establishes the threshold value conjectured by G\"ottsche in \cite{G}.

\par
In addition to knowing the value of $d$ that ensures that $N^{d,\delta}$ is given by a polynomial is, of course, the issue of determining the node polynomials exactly.  The node polynomials for the small numbers of nodes were known in the 19th century:
\begin{align*}
N_1(d) &= 3(d - 1)^2	 & &\quad\text{J.~Steiner (1848)} \\
N_2(d) &= \tfrac{3}{2}(d - 1)(d - 2)(3d^2 - 3d - 11)  & &\quad\text{A.~Cayley (1863)} \\
N_3(d) &= \tfrac{9}{2}d^6 - 27d^5 + \tfrac{9}{2}d^4 + \tfrac{423}{2}d^3 - 229d^2 - \tfrac{829}{2}d + 525
   & & \quad\text{S.~Roberts (1875)}   
\end{align*}
The node polynomials for $\delta = 4, 5, 6$ were obtained by Vainsencher in \cite{V}, for $\delta = 7, 8$ by Kleiman and Piene in \cite{KP}, and by the first author for $\delta \le 14$ in \cite{Block}.

\medskip
\par
We are particularly interested in the generating series for Severi degrees
\begin{equation} \label{Ngen}
    \SN(d)=\sum_{\delta \ge 0}N^{d,\delta} x^\delta    
\end{equation}
and its formal logarithm
\begin{equation} \label{Qgen}
      \SQ(d)=\log(\SN(d))=\sum_{\delta \ge 1}Q^{d,\delta} x^\delta .    
\end{equation}
Writing the coefficients of  $\SQ(d)$ explicitly,
\begin{equation} \label{explicit}
      Q^{d,\delta}
      =\sum\frac{(-1)^{p-1}}{p}\left(\prod_{i=1}^{p}N^{d,\delta_i}\right),   
\end{equation}
where the sum is over ordered partitions $\delta=\delta_1+\cdots+\delta_p$.  For $d$ sufficiently large and $\delta$ fixed, $N^{d, \delta}$ is given by a polynomial of degree $2\delta$.  Thus, \textit{a priori} one would expect $Q^{d,\delta}$ likewise to be a polynomial of degree $2\delta$.  However, $Q^{d,\delta}$ quite unexpectedly turns out to be \emph{quadratic}.  This is a consequence of the G\"ottsche--Yau--Zaslow Formula \cite[Conjecture~2.4]{G} (see also \cite{Qv1} and \cite{Qv2}), rather recently proved by Tzeng~\cite[Theorem~1.2]{Tz} using very sophisticated techniques. One goal of this paper is to establish the quadraticity of $Q^{d,\delta}$, for $d$ sufficiently large and fixed $\delta$, in an elementary combinatorial way.
\par
In Section \ref{leg} we describe what we call a \emph{long-edge graph}, the main combinatorial tool to determine Severi degrees. A long-edge graph is in fact nothing other than an ordered collection of \emph{templates}, as defined in \cite{FM} and \cite{Block}. They were used there to calculate Gromov--Witten invariants, Severi degrees, and node polynomials, but the perspective we take here is slightly different.  In Section \ref{laq}, we establish Theorem \ref{linearity}, which shows that a certain polynomial constructed from a long-edge graph is linear.  Then in Section \ref{templates} we discuss templates from scratch and see that the quadraticity of $Q^{d,\delta}$ follows, since it is a discrete integral of the linear polynomial of Theorem \ref{linearity}.  Finally, in Section \ref{tropical}, we explain how long-edge graphs arise from the tropical-geometric computation of Severi degrees, via the notion of \emph{floor diagrams}.
\par
One would hope to exploit the relationship between the quantities $N^{d,\delta}$  and $Q^{d,\delta}$
by inverting (\ref{Qgen}):
\begin{equation} \label{NisExpQ}
      \SN(d)=\exp(\SQ(d)).
\end{equation}
Explicitly, this gives
\begin{equation*}
      N^{d,\delta}
      =\sum\frac{1}{p!}\left(\prod_{i=1}^{p}Q^{d,\delta_i}\right),   
\end{equation*}
again summing over ordered partitions $\delta=\delta_1+\cdots+\delta_p$.  Knowing that the quantities $Q^{d,\delta}$ are quadratic in $\delta$ (and in fact obtained from certain linear quantities, as explained below), and that only templates need to be used, one should be able to efficiently calculate the Severi degrees.  What is needed is a way to calculate these quadratic quantities in some simple way from the graph-theoretic combinatorics laid out herein, rather than from the cumbersome definition (\ref{explicit}).  We intend to consider this problem further.

\par
While our formulas for $Q^{d, \delta}$ are evidently not positive, a natural question is to find an inherently positive formula for the $Q^{d, \delta}$. This would be very desirable, as it might give further insight in ``natural building blocks" of long-edge graphs and floor diagrams, in regard of identity (\ref{NisExpQ}).  We also note that, in \cite{Liu}, F.~Liu has recently and independently provided a combinatorial proof of the quadraticity of $Q^{d, \delta}$.

\par
We express appreciation to our colleagues Sergei Chmuntov, Kyungyong Lee, Boris Pittel, and Kevin Woods for their helpful comments and suggestions regarding this work.  Via the website MathOverflow, we received valuable insights into certain combinatorial issues, especially in postings by Will Sawin, Richard Stanley, Gjergji Zaimi, and David Speyer.  We thank Eduardo Esteves, Dan Edidin, Abramo Hefez, Ragni Piene, and Bernd Ulrich for arranging a most stimulating 12th ALGA Meeting and to IMPA for hosting it.  We are grateful to the referee for very useful comments that improved our exposition.  Finally, we offer our sincere gratitude to Steven Kleiman and Aron Simis for their many years of mathematical stimulation and guidance.

\section{Long-edge Graphs} \label{leg}

\par
Consider an edge-weighted multigraph $G$ on a vertex set indexed by the set of nonnegative integers $\{0,1,2,\dots\}$.   If $e$ is an edge between vertex $i$ and vertex $j$, we define the \emph{length} $l(e)$ of $e$ to be $l(e) = | i - j |$.  Denote the \emph{weight} of $e$ by $w(e)$. 

\begin{defn} \label{longedgedef}
An edge-weighted multigraph $G$ is a \emph{long-edge graph} if the following conditions hold:
\begin{enumerate}
\item  There are only finitely many edges.
\item  Multiple edges are permitted, but not loops.
\item  The weights are positive integers.
\item  The graph has no short edges, where a \emph{short edge} is an edge of length 1 and weight 1. (Thus all edges are \emph{long edges}.)
\end{enumerate}
\end{defn}

\noindent
We will draw long-edge graphs by arranging the vertices in order from left to right, with edges as segments or arcs drawn strictly from left to right, and indicating only the weights of 2 or more.  The \emph{multiplicity} $\mu$ of a long-edge graph $G$ is the product of the squares of the edge weights:
\[     \mu(G):=\prod w(e)^2.  \]
Its \emph{cogenus} is
\[     \delta(G):=\sum \left(l(e) \cdot w(e)-1\right),  \]
summing over all edges.  Our definition is inspired by the \emph{floor diagrams} of Brugall\'e and Mikhalkin \cite{BM} and Fomin and Mikhalkin's variant thereof \cite{FM}. We discuss the precise relationship in Section~\ref{tropical}.

\par
For each nonnegative integer $i$, let
\[    w_i :=\sum w(e),   \]
the sum taken over all edges lying over the interval $[i, i+1]$, i.e., edges beginning at or to the left of $i$, and ending at or to the right of $i+1$.  

\begin{defn} \label{allowdef}
Given a positive integer $d$, we say that a long-edge graph is \emph{allowable for $d$} if it satisfies these three criteria:
\begin{enumerate}
\item  All of the vertices to the right of vertex $d+1$ have degree zero.  (That is, there are no edges after vertex $d+1$.)
\item  All edges incident to vertex $d+1$, if any, have weight $1$.
\item  Each $w_i \leq i$.
\end{enumerate}
\end{defn}
Note that if a long-edge graph $G$ satisfies criterion (3) in Definition \ref{allowdef}, then there is some value of $d$ for which it is allowable, and that if $G$ is allowable for a particular value of $d$, it is allowable for all $d' > d$ as well.

\begin{example} \label{legex}
The long-edge graph $G$ shown in Figure \ref{Gexample} is allowable for all $d \ge 5$. Note that $\mu(G) = 4$ and $\delta(G) = 3$.  In addition, $w_i = 0$ for $0 \le i \le 2$, $w_3 = 1$, $w_4 = 4$, and $w_5 = 1$.  
\end{example} 
\begin{figure}[htbp]
  \centering
  \includegraphics{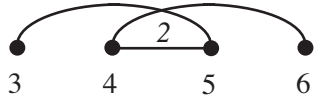} 
  \caption{The long-edge graph $G$ of Example \ref{legex}.}
  \label{Gexample}
\end{figure}

\par
If $G$ is allowable for $d$, then we obtain its \emph{extended graph} $\ext_{d}(G)$ by adding short edges to $G$ as follows:  for each $i \le d$, add $i-w_i$ such edges over the interval $[i, i+1]$.  Note that in $\ext_{d}(G)$ the number of edges over $[i, i+1]$ will be exactly $i$ (counting each edge with its multiplicity). If we subdivide each edge of $\ext_{d}(G)$ by introducing one new vertex, we obtain a graph which we denote by $G'_d$.  An \emph{ordering} of $G'_d$ is a linear ordering of its vertices that extends the ordering of the vertices $\{0,1,2,\dots\}$ of $G$. (See Figure \ref{GextGprime}.)
Two such orderings are considered \emph{equivalent} if there is an automorphism of $G'_d$ preserving the vertices of $G$. 

\par
If $G$ is allowable for $d$, then we define
\[   N^{d,G}=\mu(G) \cdot(\text{\# equivalence classes of orderings of $G'_d$}),    \]
remarking that this is independent of $d$ (as long as the graph is allowable for $d$).  If $G$ is not allowable for $d$, then let $N^{d,G}=0$.

\begin{example} \label{legexcontd}
The graphs $\ext_{5}(G)$ and $G'_5$ associated to the long-edge graph $G$ of Example \ref{legex} are shown in Figure \ref{GextGprime}.  In any ordering of $G'_5$ we require $3 < v< 5$ and $4 < w < 6$.  Thus there are $3\cdot 7 = 21$ (inequivalent) orderings if $3 < v < 4$; there are $2\cdot5 = 10$ orderings if $4 < v < 5$ and $5 < w < 6$; and there are $6$ orderings if both $v$ and $w$ are between vertices labeled $4$ and $5$.  Hence $N^{5,G} = 4(21 + 10+ 6) = 148$.
\end{example}
\begin{figure}[htbp]
  \centering
  \includegraphics{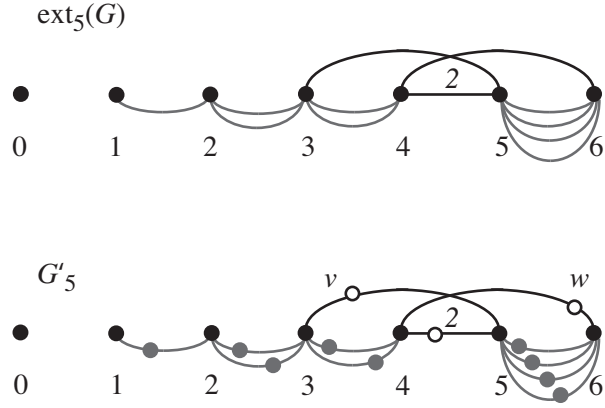} 
  \caption{The graphs $\ext_{5}(G)$ and $G'_5$ associated to the graph $G$ of Example \ref{legex} and Figure \ref{Gexample}.}
  \label{GextGprime}
\end{figure}

\par
The significance of the constructions above is that they enable a combinatorial calculation of the Severi degrees of $\PP^2$.
\begin{thm} \label{FMrecast}
The Severi degree may be computed as 
\[    N^{d,\delta} = \sum N^{d,G},    \]
where the sum is taken over all long-edge graphs of cogenus $\delta$.
\end{thm} 

\noindent
Note that, for each pair $d$, $\delta$, only finitely many terms of the sum above are nonzero.  Theorem \ref{FMrecast} is essentially a recasting of \cite[Theorem 1.6, Corollary 1.9]{FM} and \cite[Theorem 3.6]{BM};  see Theorem \ref{FMcount} below.

\par
Although $N^{d,G}$ is the quantity which enters into Theorem \ref{FMrecast},
for purposes of calculation we often find it more convenient to
work with an ``automorphism-free'' and ``multiplicity-free'' quantity.
Suppose that the edges of $G$ have been labeled. Then  (in the allowable cases) we define
$N_{*}^{d,G}$ to be the  number of orderings of $G'_d$, so that
\[      N^{d,G}=\frac{\mu(G)}{\alpha(G)}N_{*}^{d,G},   \]
where $\alpha(G)$ is the number of automorphisms of $G$ when the edges are unlabeled.  (The vertices remain labeled, however.)  See Figure \ref{autofig} for an example.  Note that the short edges added to create $G'_d$ are considered to be unlabeled. Any of these short edges which lie completely to the left or right of the edges of $G$ are irrelevant in the calculation of $N^{d,G}$; going forward, therefore, we usually will not display such edges.
\begin{figure}[htbp]
   \centering
   \includegraphics{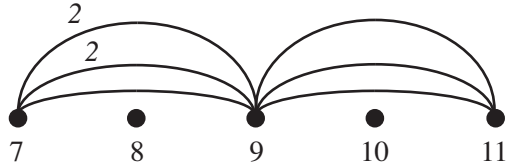} 
   \caption{For this graph $G$, we have $\mu(G) = 16$ and $\alpha(G) = 12$.}
   \label{autofig}
\end{figure}

\begin{example}\label{nd1}
We calculate $N^{d,1}$.  There are two types of long-edge graphs of cogenus one:  either the graph has a single edge of length $2$ and weight $1$, or a single edge of length $1$ and weight $2$.  They are shown in Figure \ref{CyclStub}; we call them the \emph{cyclops} and the \emph{stub}, respectively.  The cyclops $\text{Cyc}[k]$ has multiplicity $1$ and is allowable for $d$ if $1 \leq k \leq d-1$, while $\text{Stub}[k]$ has multiplicity $4$ and is allowable for $d$ if $2 \leq k \leq d-1$. There are no non-trivial automorphisms.
\begin{figure}[htbp]
   \centering
   \includegraphics{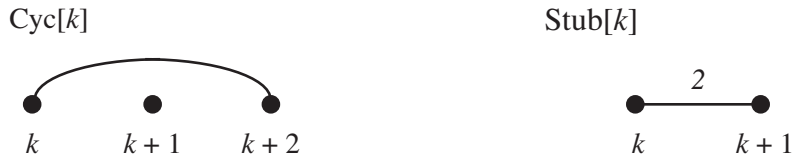} 
   \caption{The two types of long-edge graphs of cogenus $1$.}
   \label{CyclStub}
\end{figure}

\par
To obtain the extended graph $\ext_d(\text{Cyc}[k])$, we add $k-1$ short edges over the interval $[k, k+1]$, and $k$ short edges over $[k+1, k+2]$ (as well as irrelevant edges further to the left or right).  An ordering of $\text{Cyc}[k]'_d$ is determined by the position of the new vertex on the long edge, and there are $2k+1$ possible positions. (See Figure \ref{CyclStub2}.)  Similarly, $\ext_d(\text{Stub}[k])$ is obtained by adding $k-2$ short edges over $[k, k+1]$, and there are $k-1$ possible positions for the new vertex on the long edge.  Thus in the allowable cases we have
\[     N^{d,\text{Cyc}[k]} = 2k+1  \quad\text{and}\quad N^{d,\text{Stub}[k]} = 4(k-1).   \]
Hence
\[
N^{d,1} = \sum_{k=1}^{d-1} (2k+1) + \sum_{k=2}^{d-1} 4(k-1) = 3(d-1)^2. 
\]
\end{example}
\begin{figure}[htbp]
   \centering
   \includegraphics{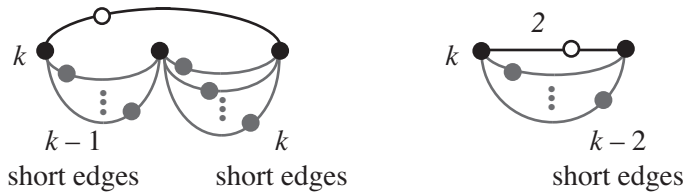} 
   \caption{Orderings of the long-edge graphs of cogenus $1$.}
   \label{CyclStub2}
\end{figure}
\par
To calculate $N^{d,G}$ for more complicated graphs, it is useful to work with  distributions of the new vertices on the long edges. A \emph{distribution} is a function $\Delta$ that associates, to each edge $e$ of $G$, one of the $l(e)$ intervals over which it lies.  We say that an ordering of $G'_d$ is \emph{consistent} with $\Delta$ if, in this ordering, each new long-edge vertex introduced by the subdivision process (as described above) lies within the interval specified by $\Delta$.  Let 
\begin{equation*}
N^{d,(G,\Delta)}=\mu(G) \cdot(\text{\# equivalence classes of orderings of $G'_d$ consistent with $\Delta$}),
\end{equation*} 
noting as before that this number is independent of $d$, as long as the graph is allowable for $d$.
Again we declare $N^{d,(G,\Delta)}=0$ when $G$ is not allowable for $d$.
Summing over all possible distributions, we have
\[    N^{d,G} = \sum_{\Delta} N^{d,(G,\Delta)}.   \]
\par
As above we often find it more convenient
to work with the automorphism- and multiplicity-free quantity 
$N_{*}^{d,(G,\Delta)}$, noting that
\[      N^{d,(G,\Delta)}=\frac{\mu(G)}{\alpha(G,\Delta)}N_{*}^{d,(G,\Delta)},   \]
where $\alpha(G,\Delta)$ is 
the number of automorphisms of $G$ consistent with $\Delta$.
Since the short edges added to create $G'_d$ are considered to be
unlabeled and therefore indistinguishable (when they lie over the same interval),
we have
\begin{equation} \label{falling}
N_{*}^{d,(G,\Delta)}=\prod_{i}(i-w_i+m_i)_{m_i},
\end{equation}
where $m_i$ is the number of times that $[i,i+1]$
appears as a value of $\Delta$, and
where $(i-w_i+m_i)_{m_i}$ indicates a falling factorial (i.e., $(a)_m = a(a-1)\cdots (a-m+1)$ and we take $(a)_0$ to be $1$).  The product in formula (\ref{falling}) is taken over all $i \ge 1$; however, all but finitely many factors have value $1$.

\par
If we translate a long-edge graph $G$ rightward by $k$ units, we obtain another long-edge graph $G[k]$, which we will call an \emph{offset} of $G$. In Example \ref{nd1}, the graphs $\text{Cyc}[k]$ and $\text{Stub}[k]$ are offsets of the graphs shown in Figure \ref{CyclStubtemp}, which we call the \emph{cyclops template} and the \emph{stub template}.  (The general notion of a template is explained in Section \ref{templates}.  The nomenclature originates with \cite{FM}.)
If $\Delta$ is a distribution of $G$, then $\Delta[k]$ is the distribution of $G[k]$ defined in the obvious way:  if $\Delta(e)=[i,i+1]$ then $\Delta(e[k])=[k+i,k+i+1]$.  Note that, for any $G$, we may choose a sufficiently large offset $k$ so that $G[k]$ satisfies criterion (3) of Definition \ref{allowdef}.
\begin{figure}[htbp]
   \centering
   \includegraphics{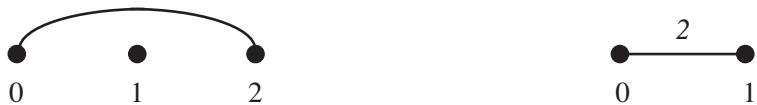} 
   \caption{The cyclops and stub templates.}
   \label{CyclStubtemp}
\end{figure}

\begin{prop} \label{monicN}
$N_{*}^{d,(G[k],\Delta[k])}$  is a monic polynomial in $k$ for sufficiently large $k$. Its degree
is the number of edges of $G$.
\end{prop}
\begin{proof}
By formula (\ref{falling}) we have
\[      N_{*}^{d,(G[k],\Delta[k])}=\prod_{i}(k+i-w_i+m_i)_{m_i}.   \]
\end{proof}

\section{Linearity} \label{laq}

\par
Let $G$ be a long-edge graph satisfying criterion (3) of Definition \ref{allowdef}; let $n$ be the number of edges of $G$ and
let $\Edge(G)$ denote the set of edges.  For each subset $E$ of $\Edge(G)$, consider the subgraph with these edges; for simplicity we also denote it by $E$. Note that any distribution $\Delta$ is inherited by $E$.  We now consider, for each $d$, the alternating sums
\begin{equation}\label{qg}
      Q^{d,G}=\frac{1}{\alpha(G)}\sum_{\SP}(-1)^{p-1}(p-1)!\prod_{E \in \SP}\left(\alpha(E) N^{d,E}\right)    
\end{equation}
and
\begin{equation} \label{qgdelta}
      Q^{d,(G,\Delta)}=\frac{1}{\alpha(G,\Delta)}\sum_{\SP}(-1)^{p-1}(p-1)!
      \prod_{E \in \SP}\left(\alpha(E,\Delta)N^{d,(E,\Delta)}\right),
\end{equation}
summing in both instances over all unordered partitions $\SP$ of  $\Edge(G)$, taking products over the blocks $E$ of $\SP$, and denoting by $p$ the number of blocks.  In view of Proposition \ref{monicN}, we know that $Q^{d,(G,\Delta)}$ is a polynomial whose degree is at most $n$.  In this section we show that, surprisingly, it is linear.

\par
The automorphisms make the formulas in (\ref{qg}) and (\ref{qgdelta}) look somewhat awkward, but if we use instead the automorphism- and multiplicity-free quantity
\[      Q_{*}^{d,(G,\Delta)}=\frac{\alpha(G,\Delta)}{\mu(G)}Q^{d,(G,\Delta)},   \]
then (\ref{qgdelta}) becomes
\begin{equation}   \label{qstar} 
Q_{*}^{d,(G,\Delta)}=\sum_{\SP}(-1)^{p-1}(p-1)!
      \prod_{E \in \SP}N_{*}^{d,(E,\Delta)}.
\end{equation}

\par
To provide some motivation for considering the particular alternating sums in (\ref{qg}) and (\ref{qgdelta}), we show how they allow us to refine the generating series (\ref{Ngen}) and (\ref{Qgen}).  The \emph{disjoint union} of the long-edge graphs $G_1, G_2, \dots$ is the graph $\sqcup G_i$ obtained by taking the disjoint union of their edge sets. Note that the cogenus $\delta(\sqcup G_i)$ is the sum $\sum \delta(G_i)$.  Introducing a formal indeterminate $x^G$ for each long-edge graph $G$, let
\begin{equation} \label{bseries}
\SNb(d)=\sum N^{d,G} x^G \qquad\text{and}\qquad  \SQb(d)=\log{\left(\SNb(d)\right)} =\sum Q^{d,G} x^G,  
\end{equation}
summing over all long-edge graphs $G$.  Here we take $\prod x^{G_i}$ to mean $x^{\sqcup G_i}$.  Equating the coefficients in (\ref{bseries}) yields (\ref{qg}).  
Theorem~\ref{FMrecast} tells us that 
G{\"o}ttsche's generating series $\SN$ 
can be recovered from 
from $\SNb$ by replacing each $x^G$ by $x^{\delta(G)}$.
Thus the same is true for their logarithms: $\SQ$
can be recovered from 
 $\SQb$ by the same replacement. This means that 
\[        Q^{d,\delta} = \sum_{G} Q^{d,G} ,    \]
summing over all long-edge graphs of cogenus $\delta$.
\par
We may refine further by taking into account the distributions: let 
\begin{equation} \label{dseries}
\SNbb(d)=\sum N^{d,(G,\Delta)} x^{(G,\Delta)} \quad\text{and}\quad  \SQbb(d)=\log{\left(\SNbb(d)\right)} =\sum Q^{d,(G,\Delta)} x^{(G,\Delta)}, 
\end{equation}
so that  (\ref{qgdelta}) is the result of equating coefficients. 
Since the generating series $\SNb$ 
can be recovered  
from $\SNbb$ by replacing each $x^{(G,\Delta)}$ by $x^G$,
the same replacement takes $\SQbb$ 
to $\SQb$. This means that 
\[        Q^{d,G} = \sum_{\Delta} Q^{d,(G,\Delta)},    \]
summing here over all possible distributions for $G$.


\begin{example} \label{QCalc}
We illustrate the calculation of $Q^{d,G}$ for the graph $G$ shown in Figure \ref{LTemplate}, assuming that the graph is allowable for $d$.  
(Explicitly, we assume that $k \geq 4$ and $d \geq k+1$).  Note that $w_k = 4$, $w_{k+1} = 2$, and $\mu(G) = 4$.  
There are three possible distributions of subdivision points, illustrated in Figure \ref{Subdivision}, with automorphisms as indicated there.  Thus
\[      Q^{d,G}=Q^{d,(G,\Delta_{1})}+Q^{d,(G,\Delta_{2})}+Q^{d,(G,\Delta_{3})}
=2Q_{*}^{d,(G,\Delta_{1})}+4Q_{*}^{d,(G,\Delta_{2})}+2Q_{*}^{d,(G,\Delta_{3})}.   \]
\begin{figure}[htbp]
   \centering
   \includegraphics{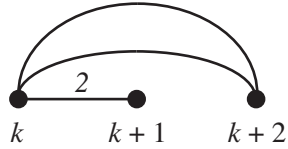}
   \caption{The graph $G$ of Example \ref{QCalc}.}
   \label{LTemplate}
\end{figure}
\begin{figure}[htbp]
   \centering
   \includegraphics{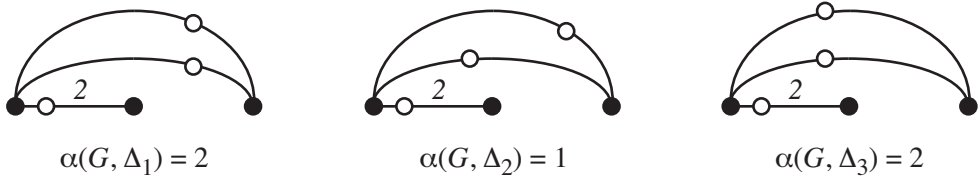}
   \caption{The three distributions of $G$.}
   \label{Subdivision}
\end{figure}
Labeling the three edges of $G$ by $A$, $B$, $C$ as in Figure \ref{TempDist1}, we have
\[   \begin{split}
Q_{*}^{d,(G,\Delta_{1})}
&=  N_{*}^{d,(G,\Delta_{1})} 
- N_{*}^{d,(A\cup B,\Delta_{1})}N_{*}^{d,(C,\Delta_{1})}  
- N_{*}^{d,(A\cup C,\Delta_{1})}N_{*}^{d,(B,\Delta_{1})} \\
&\quad\mbox{} - N_{*}^{d,(B\cup C,\Delta_{1})}N_{*}^{d,(A,\Delta_{1})} 
   +2N_{*}^{d,(A,\Delta_{1})}N_{*}^{d,(B,\Delta_{1})}N_{*}^{d,(C,\Delta_{1})} \\
&= (k-3)(k+1)_2  - (k+1)_2\cdot(k-1) - 2(k-2)(k+1)\cdot(k+1) \\
&\quad\mbox{} + 2(k+1)\cdot(k+1)\cdot(k-1) = 2k+2.
\end{split} \]
\begin{figure}[htbp]
   \centering
   \includegraphics{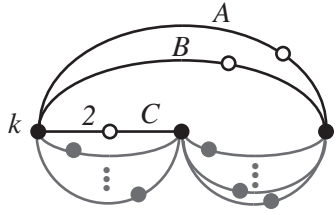}
   \caption{Figure for the computation of  $Q_{*}^{d,(G,\Delta_{1})}$.}
   \label{TempDist1}
\end{figure}
\noindent
Similarly, as illustrated by Figure \ref{TempDist2}, we have 
\begin{align*}
Q_{*}^{d,(G,\Delta_{2})} 
&= (k-2)_2 k - k(k-1)\cdot(k-1) - (k-2)(k+1)\cdot k - (k-1)_2\cdot(k+1)  \\
& \quad\mbox{}+ 2(k+1)\cdot k\cdot(k-1)  = 6k-2; \\
Q_{*}^{d,(G,\Delta_{3})} 
&=   (k-1)_3 - (k)_2\cdot(k-1) - 2(k-1)_2\cdot k + 2 k\cdot k \cdot(k-1)  \\
&= 6k-6.
\end{align*}
\noindent
Putting these results together, we find that $Q^{d,G}=40k-16$ when $k \geq 4$ (and $d$ is sufficiently large).
\begin{figure}[htbp]
   \centering
   \includegraphics{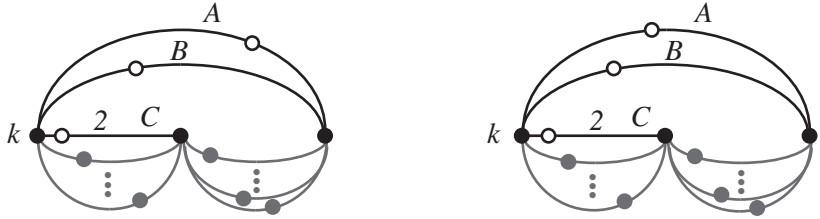}
   \caption{Diagrams for the calculations of $Q_{*}^{d,(G,\Delta_{2})}$ and 
   $Q_{*}^{d,(G,\Delta_{3})}$.}
   \label{TempDist2}
\end{figure}
\par
When $k$ is 0 or 1, then every term in the computation involves a subgraph that is not allowable, so that $Q^{d,G}=0$ in these cases.  When $k=3$, all proper subgraphs are allowable, so that only one term in the calculation is suppressed; here $Q^{d,G}=104$ (which agrees with the general formula, although this appears to be a coincidence).  When $k=2$, only two of the five partitions contribute to the calculation of $Q^{d,G}=76$.
\end{example}

\begin{thm} \label{linearity}
For each long-edge graph $G$ and each distribution $\Delta$, the polynomial 
$Q_{*}^{d,(G[k],\Delta[k])}$
 is linear in $k$ for $k$ sufficiently large.  Thus $Q^{d,G[k]}$ is likewise linear in $k$ for $k$ sufficiently large.
\end{thm}

\begin{proof}
Again let $n$ be the number of edges of $G$.  For $n =1$, the statement is clear.  Thus we assume that $n \ge 2$.  Translating to the right if necessary, we may assume that $G$ satisfies criterion (3) of Definition \ref{allowdef}.  Fix a value of $d$ for which $G$ is allowable.  Consider the extended graph $\ext_{d}(\emptyset)$ associated to the edge-less graph: it has  
$i$ short edges over each interval $[i,i+1]$, as $i$ runs from $1$ to $d$.  (See Figure \ref{extempty}.)
\begin{figure}[htbp]
   \centering
   \includegraphics{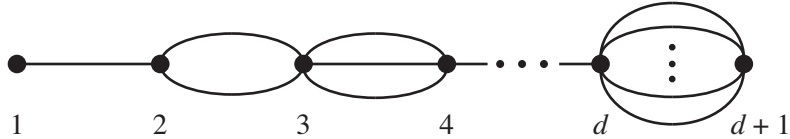}
   \caption{The extended graph $\ext_{d}(\emptyset)$ associated to the graph with no edges.}
   \label{extempty}
\end{figure}
Let $S$ be this set of short edges.  To each edge $e$ of $G$ we associate a subset $S_e \subset S$ consisting of $w(e)$ short edges over each interval covered by $e$, except over the interval $\Delta(e)$, where we take only $w(e)-1$ edges.  Note that over the interval $[i,i+1]$ we require a total of $w_i - m_i$ edges.  Thus, by criterion (3) of Definition \ref{allowdef}, these subsets can be chosen to be disjoint.  Let $S_0$ be $S \setminus \bigcup_{e \in G}S_e$.
Figure \ref{propex} presents an example.
\begin{figure}
   \centering
    \includegraphics{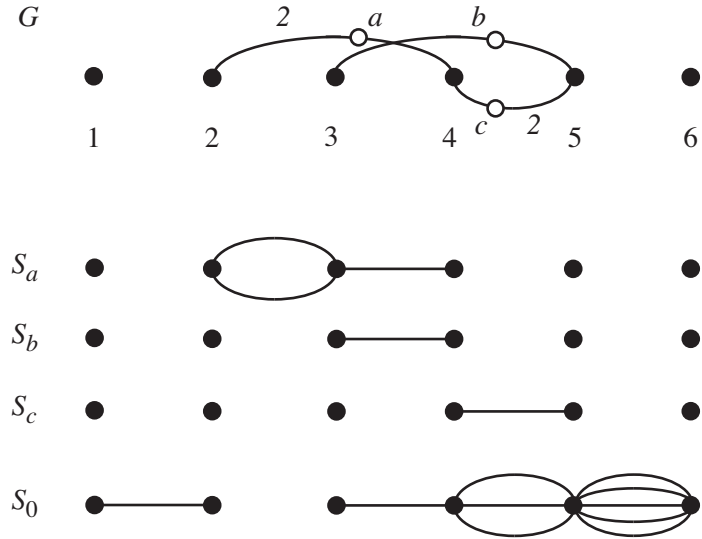}
   \caption{An example to illustrate the subsets of $S$ associated to the edges of a long-edge graph.  We are assuming $d = 5$.}
   \label{propex}
\end{figure}
\par
Then for any subset $E$ of $\Edge(G)$, the recipe for creating $\ext_{d}(E)$ amounts to this:  add to $E$ the edges of
\[     S \setminus \bigcup_{e \in E}S_e,    \]
minus one short edge over each interval $\Delta(e)$.  An ordering of $E'_d$ can be identified with an injection
\[     f:E \to S \setminus \bigcup_{e \in E}S_e   \]
for which $f(e)$ is one of the edges over $\Delta(e)$.  Thus for any partition $\SP$ of $\Edge(G)$, the product $\prod_{E \in \SP}N_*^{d,(E,\Delta)}$  counts functions $f$ from $\Edge(G)$ to $S$ having the following properties:
\begin{enumerate}
\item For each edge, $f(e)$ is one of the edges over $\Delta(e)$.
\item For each block $E$ of the partition, $f(E)$ is contained in $S \setminus \bigcup_{e \in E}S_e$.
\item On each block, $f$ is injective.
\end{enumerate}

\par
Applying this observation in (\ref{qstar}),
we can regard $Q_{*}^{d,(G,\Delta)}$ as a sum 
\begin{equation*} 
      Q_{*}^{d,(G,\Delta)}=\sum_{f}\sum_{\SP}(-1)^{p-1}(p-1)!
\end{equation*}
over functions satisfying the first condition, where in the inner sum we allow only those partitions that meet the other two conditions. We will call them \emph{compatible partitions}.  Letting $\Sigma(f)$ denote the contribution of $f$ to $Q_{*}^{d,(G,\Delta)}$,  note that $|\Sigma(f)|\leq C$, where
\[    
C=\sum_{p=1}^n(p-1)! \cdot (\text{\# $p$-block partitions of an $n$-element set}).
\]
(Recall that $n$ denotes the number of edges of $G$.)
\par
We first examine the case where $f$ is injective on the entire edge set of $G$.  Create a new \emph{auxiliary graph} $H$ as follows: take one vertex $\bar{e}$ for each edge $e$ of $G$; if $f(e_1)\in S_{e_2}$ then draw an edge between $\bar{e}_1$ and $\bar{e}_2$ (in particular if $f(e_1)\in S_{e_1}$, then draw a loop); replace any double edges by single edges.  By condition (2), $\SP$ is a compatible partition for $f$ if and only if no block of the corresponding vertex partition of $H$ contains two adjacent vertices; we say that $\SP$ is \emph{compatible} with $H$.  (Note the resemblance to the graph-theoretic notion of a coloring.  Also note that if $H$ has any loops, then no partition will be compatible.)  In Figure \ref{graphH}, we give an example to illustrate how $H$ is constructed. The graph depicted there has just two compatible partitions:  the fine partition and the partition $\{ \bar{a}, \bar{c} \} \cup \{ \bar{b} \}$.
\begin{lem} \label{graphlemma}
Suppose that $H$ is a graph on $n$ vertices with $n \ge 2$. If $H$ has at most $n-2$ edges, then
\[       \Sigma(H) := \sum_{\SP} (-1)^{p-1}(p - 1)!=0,     \]
where the sum is taken over all compatible unordered partitions of the vertex set of $H$, and $p$ is the number of blocks.
\end{lem}
\begin{figure}
   \centering
    \includegraphics{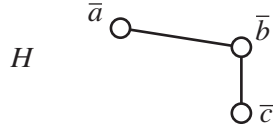}
\caption{This continues the example from Figure \ref{propex}, assuming that the function $f$ satisfies these conditions: $f(a)\in S_b$, $f(b)\in S_c$, $f(c)\in S_0$.  Note that Lemma \ref{graphlemma} does not apply to $H$ since $H$ has more than a single edge.  As $G$ is offset, however, only the set $S_0$ increases in cardinality and so the number of such functions only grows linearly with the offset $k$.}
   \label{graphH}
\end{figure}
\begin{proof}
If $H$ has any loops, then $\Sigma(H) = 0$.  If $H$ has no loops, then we use induction on the number of edges. If $H$ has no edges, then the equation expresses a standard identity of the Stirling numbers of the first and second kinds. (See \cite[Proposition 1.9.1]{Stanley}, recalling that the Stirling number of the first kind $s(p,1) = (-1)^{p-1}(p-1)!$.)  Otherwise chose an edge $e$. Let $H'$ be the graph obtained by omitting it, and let $H''$ be the graph obtained by identifying its two vertices $v$ and $w$ (and then removing the loop and any redundant edges).  Let $\SP$ be a vertex partition compatible with $H'$. It is compatible with $H$ if and only if $v$ and $w$ belong to different blocks of it.  If $v$ and $w$ belong to the same block of $\SP$, then we obtain a compatible vertex partition of $H''$, and, moreover, one obtains all compatible partitions of $H''$ in this way. Note that both $H'$ and $H''$ both have fewer edges than $H$. Thus
\[ 
\Sigma(H) = \Sigma(H') - \Sigma(H'') = 0.   
\]
\end{proof}

\begin{rem}
S.\ Chmutov has pointed out to us that the polynomial $\Sigma(H)$ in Lemma \ref{graphlemma} is the value at $p = 0$ of the derivative of the chromatic polynomial $C_H(p)$.  To see this, note first that for a graph $H$ with $k$ connected components, the chromatic polynomial is divisible by $p^k$.  Hence, $C_H'(0) = 0$ if $H$ is disconnected.  Moreover, with the graphs $H'$, $H''$ defined as in the proof of Lemma \ref{graphlemma}, we have the recurrence relation $C_H(p) = C_{H'}(p) - C_{H''}(p)$.  Thus the computation of $C_H'(0)$ for a connected graph $H$ reduces to the calculation of $C_G'(0)$ where the graph $G$ consists of a single point.  But $C_G(p) = p$, so that $C_G'(0) = 1$, which agrees with $\Sigma(G)$.
\end{rem}

\par
Returning to the proof of Theorem \ref{linearity}, note that for an injection $f$ we have $\Sigma(f)=\Sigma(H)$, where $H$ is the auxiliary graph.  Also note that the number of edges in $H$ is bounded above by $n$ minus the number of values of $f$ which lie in $S_0$.  Thus, by Lemma \ref{graphlemma}, we see that for an injection satisfying properties (1), (2), and (3) we have $\Sigma(f)=0$ except in those cases where at most one of the values of $f$ lies in $S_0$.  We claim that the same is true for any function satisfying properties (1), (2), and (3), and prove this claim by induction on the number of repeated values, by which we mean
\[        r=n-\#\operatorname{Im}(f).       \]
If $r=0$ then $f$ is injective. Otherwise there is a pair of edges $e_1$, $e_2$ of $G$ for which $f(e_1)=f(e_2)$. Define two new functions $f'$ and $f''$ as follows. Suppose that $f(e_1) \in S_{e}$ (where $e$ is either an edge of $G$ or the value $0$).
Define $f'$ to be the same as $f$ except that $f'(e_2)$ is redefined to be some other element of $S_{e}$ not in the image of $f$ (i.e., different from all other values). If there is no such unused element in $S_e$, we simply enlarge $S_e$ (and hence $S$) by throwing in one more element.  To define $f''$, let $\Edge(G)''$ be the set obtained from $\Edge(G)$ by identifying $e_1$ and $e_2$ to a single element $\star$; then $f$ factors through the quotient map $\Edge(G) \to \Edge(G)''$ followed by $f'':\Edge(G)'' \to S$. Let $S_{\star}=S_{e_1}\cup S_{e_2}$.  Note that for both $f'$ and $f''$ the value of $r$ has decreased.

\par
Now observe that any partition compatible with $f$ is likewise compatible with $f'$.  Going the other way, if $\SP$ is compatible with $f'$ then there are two possibilities:  (1) $e_1$ and $e_2$ belong to different blocks, so that $\SP$ is also compatible with $f$, or (2) $e_1$ and $e_2$ belong to the same block, so that $\SP$ comes from a partition of $\Edge(G)''$ compatible with $f''$. Thus 
\[ 
\Sigma(f) = \Sigma(f') - \Sigma(f'') = 0.   
\]
This completes the proof of the claim.
\par
Finally we note that, as the offset $k$ varies, 
the sets $S_e$ associated to the edges of $G[k]$ stay the same size, while the size of $S_0$ grows linearly. Thus the number of functions having at most one of their values in $S_0$ is bounded by a linear function of $k$.
The contribution $\Sigma(f)$ of each such function 
to $Q_{*}^{d,(G[k],\Delta[k])}$
is bounded by the constant $C$ which depends only
on the number of edges in $G$, and is thus independent of $k$.
 Thus the polynomial $Q_{*}^{d,(G[k],\Delta[k])}$ is linear in $k$.
\end{proof}

\section{Templates and Quadraticity of $Q^{d, \delta}$} \label{templates}

\par
We have already encountered examples of templates in Section \ref{leg}.  Now we provide the formal definition. It is inspired by \cite[Definition 5.6]{FM}, where the term template was coined.

\begin{defn} \label{tempdef}
The \emph{right end} of a long-edge graph $G$ is the smallest vertex for which all vertices to the right have degree $0$. A vertex between vertex $0$ and the right end is called an \emph{internal vertex}. An internal vertex is said to be \emph{covered} if there is an edge beginning to the left of it and ending to the right of it. A nonempty long-edge graph $G$ is called a \emph{template} if every internal vertex is covered. The offset graph $G[k]$ of a template $G$ is called an \emph{offset template}.
\end{defn}

\noindent
Figure \ref{tempexamples} shows an example of two long-edge graphs, one a template, and the other not.
Note, in particular, that in a template the vertex $0$ has nonzero degree (and thus a template is never an allowable graph).
\begin{figure}[htbp]
   \centering
   \includegraphics{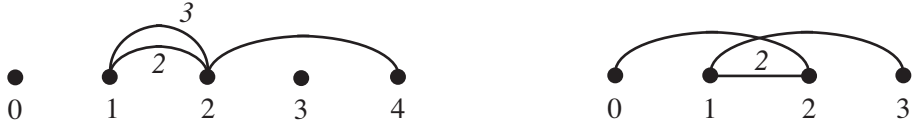} 
   \caption{The long-edge graph on the left is \emph{not} a template:  the internal vertices labeled $1$ and $2$ are not covered.   The long-edge graph on the right is a template.}
   \label{tempexamples}
\end{figure}

\begin{lem} \label{lemA}
Each long-edge graph can be expressed in a unique way as a disjoint union of offset templates.
\end{lem}

\begin{proof}
Break the graph at each non-covered vertex.
\end{proof}

\begin{lem} \label{lemB}
Given $\delta > 0$, there are finitely many templates $\Gamma$ with cogenus $\delta$.
\end{lem}

\begin{proof}
Since $\delta(\Gamma) = \sum \left(l(e) \cdot w(e)-1\right)$, there are at most $\delta$ edges, and there is an evident limit on the length and weight of each edge.
\end{proof}

\begin{prop} \label{vanishprop}
If a long-edge graph $G$ is not an offset template, then $Q_{*}^{d,(G,\Delta)}=0$.  Hence $Q^{d,(G,\Delta)}=0$.
\end{prop}

\begin{proof}
Since $G$ is not an offset template, there must be an internal vertex $v$ that fails to be covered by an edge of $G$.  Thus $v$ breaks $G$ into two subgraphs $G_\text{left}$ and $G_\text{right}$, so that 
\[
N_{*}^{d,(G,\Delta)}  = N_{*}^{d,(G_\text{left},\Delta)} N_{*}^{d,(G_\text{right},\Delta)}.
\]
Given any partition $\SP$ of the edge set of $G$, we obtain partitions $\SP_\text{left}$ and $\SP_\text{right}$ of the edge sets of $G_\text{left}$ and $G_\text{right}$.  The blocks of $\SP_\text{left}$ are the nonempty subsets $E \cap \text{(edge set of $G_\text{left}$)}$, where $E$ is a block of $\SP$. We say that $\SP$ is \emph{consistent} with $\SP_\text{left}$ and $\SP_\text{right}$.  We have
\begin{equation} \label{Qbreak}
Q_{*}^{d,(G,\Delta)} = \sum_{\SP_\text{left}} \sum_{\SP_\text{right}} \sum_{\text{consistent $\SP$}} (-1)^{p-1}(p-1)! \prod_{E \in \SP}N_{*}^{d,(E,\Delta)}.
\end{equation}
\par
We call $\SP$ \emph{allowable for $d$} if every one of its blocks is allowable for $d$;  if $\SP$ is not allowable for $d$, then
\[
\prod_{E \in \SP}N_{*}^{d,(E,\Delta)}=0.
\]
Now note that $\SP$ is allowable for $d$ if and only if both $\SP_\text{left}$ and $\SP_\text{right}$ are allowable for $d$.
Thus in the sum of (\ref{Qbreak}) we need only consider the terms in which both $\SP_\text{left}$ and $\SP_\text{right}$
are allowable.
\par
Fixing $\SP_\text{left}$ and $\SP_\text{right}$, consider
\begin{equation*} 
\sum_{\text{consistent $\SP$}} (-1)^{p-1}(p-1)! \prod_{E \in \SP}N_{*}^{d,(E,\Delta)},
\end{equation*}
which is the constant
\[
\prod_{E_\text{left} \in \SP_\text{left}}N_{*}^{d,(E_\text{left},\Delta)} 
\prod_{E_\text{right} \in \SP_\text{right}}N_{*}^{d,(E_\text{right},\Delta)}
\]
times the alternating sum
\[
\sum_{\text{consistent $\SP$}} (-1)^{p-1}(p-1)! .
\]
Let $a$ and $b$ denote the numbers of blocks of $\SP_\text{left}$ and $\SP_\text{right}$, respectively, and set $q :=  a + b - p$.  Then the coefficient of $\prod_{\text{blocks $E$ of $\SP$}}N_{*}^{d,(E,\Delta)}$ is
\begin{equation} \begin{split} \label{identity}
\sum_{q=0}^{\min(a,b)}&(-1)^{a + b - q- 1}(a + b -q-1)! \cdot (\text{\# $p$-block  $\SP$'s consistent with $\SP_\text{left}$, $\SP_\text{right}$}) \\ \notag
&= \sum_{q=0}^{\min(a,b)}(-1)^{a + b -q-1}(a + b -q-1)!\binom{a}{q}\binom{b}{q}q!. 
\end{split} \end{equation}
We prove that this evaluates to zero. Consider the two sets $A = \{x_1,\dots, x_a \}$ and $B = \{y_1,\dots, y_b \}$, and pair $q$ elements from each.  This can be done in $\binom{a}{q}\binom{b}{q}q!$ ways.  Construct $(a+b-q)$ subsets of the disjoint union $A \sqcup B$, each of which is either a singleton or a pair of the form $\{x_i, y_j \}$ where $x_i \in A$, $y_j \in B$, and arrange them in order, always beginning with the subset containing $x_1$.  (Equivalently, arrange in order up to a cyclic permutation of the subsets.)  Then the number of such ordered subsets of $A \sqcup B$ is $(a + b -q-1)!\binom{a}{q}\binom{b}{q}q!$; call this set of arranged subsets $\calS$.  We define a bijection from $\calS$ to itself as follows.  Given an element of $\calS$, read it in order (with the subset containing $x_1$ always first).  Identify the first position where there is either a pair, or an element of $A$ that is immediately followed by an element of $B$.  In the first case, replace the pair $\{x_i, y_j \}$ with $\{x_i\}$, $\{y_j \}$; in the second case, replace $\{x_i\}$, $\{y_j \}$ with $\{x_i, y_j \}$.  
Note that this bijection changes the parity of $q$.
Thus
\[   \sum_{q \text{ even}} (-1)^{a + b -q-1}(a + b -q-1)!\binom{a}{q}\binom{b}{q}q!  = \sum_{q \text{ odd}} (-1)^{a + b -q-1}(a + b -q-1)!\binom{a}{q}\binom{b}{q}q!.   \]
\end{proof}

\begin{thm}
For each $\delta$, the polynomial $Q^{d,\delta}$ is quadratic  in $d$ for $d$ sufficiently large. 
\end{thm}

\begin{proof}
By Lemma \ref{lemB} and Proposition \ref{vanishprop}, we may write
\[       Q^{d,\delta}=\sum_{\Gamma} \sum_{k}Q^{d,\Gamma[k]},    \]
a sum over the finitely many templates $\Gamma$ of cogenus $\delta$ and over all $k$ for which $\Gamma[k]$ is allowable for $d$.  For each such template, as $d$ varies the inner sum begins at a fixed lower limit and ends at an upper limit which is linear in $d$.  Furthermore the terms are linear in $k$ for $k$ sufficiently large.  Thus each inner sum is quadratic in $d$ for $d$ sufficiently large, and the same is true of the whole sum.
\end{proof}

\section{From Tropical Curves to Long-edge Graphs, via Floor Diagrams} \label{tropical}

\par
In Section~\ref{leg} we defined long-edge graphs, and in Theorem~\ref{FMrecast} we asserted that one may compute the Severi degree by computing a certain sum over such graphs. Here we explain how these long-edge graphs arise, and explicate a proof of Theorem~\ref{FMrecast}.  Our route is through tropical geometry and the theory of floor diagrams, building on the work in \cite{BM} and \cite{FM}.  We assume a familiarity with the basic notions of tropical plane curves.  (See especially these two papers for treatments related to the present context.)  By Mikhalkin's Correspondence Theorem \cite[Theorem 1]{Mik}, the classical Severi degree $N^{d,\delta}$ is the same as its tropical counterpart.

\par
Let $\mathcal{T}$ be a tropical plane curve passing through a \emph{tropically generic} point configuration (see \cite[ Definition 4.7]{Mik}). We create an \emph{associated graph} (in fact a weighted
directed multigraph) in the following manner (see Figure \ref{tropfloor} for an example).  Define an \emph{elevator} of $\mathcal {T}$ to be any vertical edge, i.e., any edge parallel to the vector $(0,1)$. The \emph{multiplicity} of an elevator is inherited from the multiplicity of that edge in the tropical curve. A \emph{floor} of $\mathcal {T}$ is a connected component of the union of all nonvertical edges. Note that elevators may cross floors. We contract each floor to a point, creating the vertices of a graph. The directed edges of this graph correspond to the elevators, with their directions corresponding to the downward (i.e., $(0,-1)$-) direction of the elevators.
For a curve of degree $d$ there will be $d$ unbounded elevators, all of multiplicity $1$, that we make adjacent to one additional vertex.   Note that the \emph{divergence}
\[ 
\operatorname{div}{(v)} := 
\sum_{\substack{\text{outward edges} \\ \text{from }v}} w(e) 
- \sum_{\substack{\text{inward edges} \\ \text{to }v}} w(e)
\]
has value $1$ at each vertex $v$ except the additional vertex, where the value is $-d$.
If the tropical curve passes through a \emph{vertically stretched} point configuration (see \cite[Definition 3.4]{FM}) then what we have just defined is virtually the same as a \emph{floor diagram}, as defined in \cite[Section~1]{FM} (c.f.\ also \cite[Section~5.2]{BM});
the corresponding floor diagram simply omits the additional vertex and its $d$ adjacent edges and carries a linear order on the remaining vertices; see Figure \ref{tropfloor} again.
\begin{figure}[htbp]
   \centering
   \includegraphics{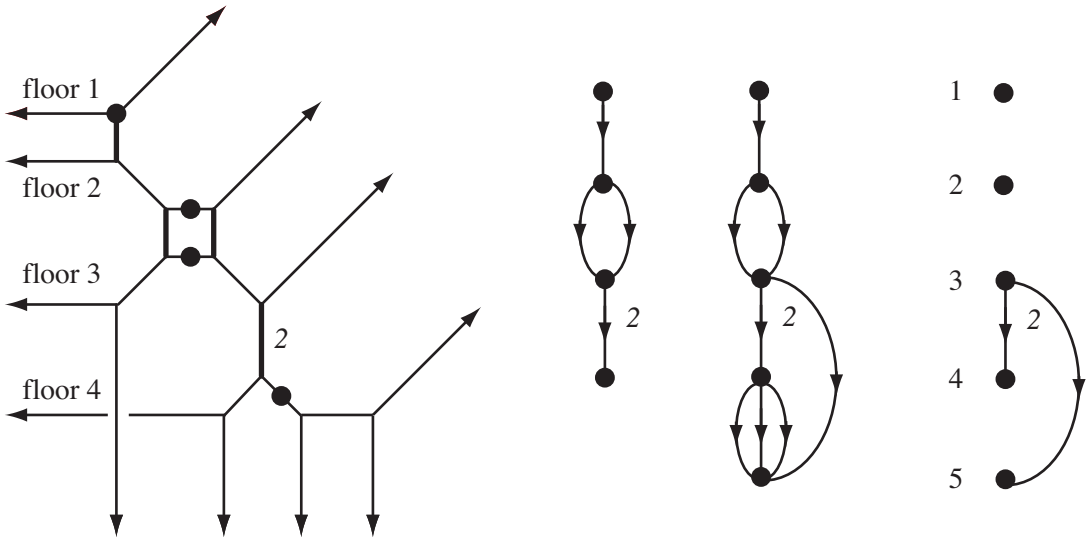}
   \caption{A plane tropical quartic, with its floor diagram, its associated graph, and its long-edge graph.}
   \label{tropfloor}
\end{figure}

\par
To obtain a long-edge graph from the associated graph, we would first like to order the vertices so that each edge goes from a smaller vertex to a larger one. In general this is impossible however, as shown in the example of Figure \ref{escher}.
Fomin and Mikhalkin \cite[Theorem 3.7]{FM} show (c.f.\ also \cite[Lemma~5.7]{BM}), however, that if the specified point conditions are vertically stretched, then, for each tropical curve of specified genus satisfying these point conditions, one indeed obtains a floor diagram (with edge directions respecting the linear order of the floors).  Thus, by adding the additional vertex (giving it the label $d+1$) and its $d$ incident edges, we obtain the associated graph.  Erasing all short edges (those of weight 1 and length 1), we then get a long-edge graph.  In the other direction, beginning with a long-edge graph, we can draw short edges so that
 $\operatorname{div}{(v)}  = 1$, and then erase vertex $d+1$ and its incident edges.
\begin{figure}[htbp]
   \centering
   \includegraphics{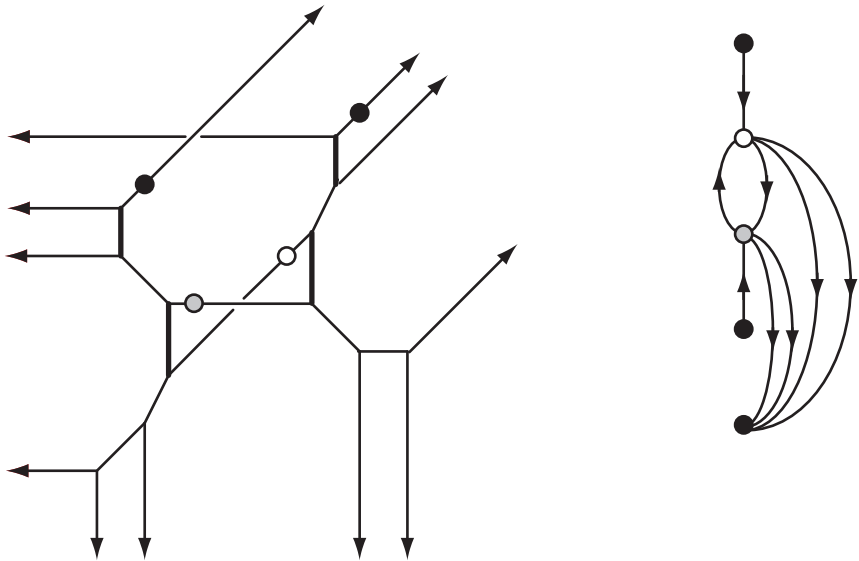}
   \caption{An ``Escher-like" tropical plane quartic with its associated graph. There is no consistent way to number the vertices.}
   \label{escher}
\end{figure}
\par
The \emph{cogenus} $\delta$ of a connected labeled floor diagram $\SD$
is $\delta=\frac{(d-1)(d-2)}{2} - g$,  where $g$ denotes the genus of its underlying graph; if $\SD$ is not connected, then
\[      \delta = \sum_j \delta_j + \sum_{j < j'}d_j d_{j'},     \]
where the $d_j$'s and $\delta_j$'s are the respective degrees (i.e., the number of vertices) and cogenera of the connected components.  The \emph{multiplicity} $\mu$ of $\SD$ is 
$\mu(\SD) = \prod_{\text{edges $e$}} (w(e))^2$. 
These definitions are compatible with the earlier definitions for long-edge graphs.
Now suppose that $G$ is the long-edge graph obtained from the labeled floor diagram $\SD$ by the process just described. Then a \emph{marking} of $\SD$, as defined in \cite{Block} and \cite {FM}, is equivalent to an ordering of $G'_{d}$, as defined in Section~\ref{leg}.  
Let $\nu(\SD)$ denote the number of equivalence classes of markings of $\SD$. (Two markings are \emph{equivalent} if they differ by a vertex and edge-weight preserving graph automorphism.)
\noindent
\begin{thm}[\cite{FM}, Theorem 1.6, Corollary 1.9]  \label{FMcount}
The Severi degrees are given by
\[   N^{d,\delta} = \sum \mu(\SD) \nu(\SD),   \]
where the sum is taken over all labeled floor diagrams (not necessarily connected) of degree $d$ and cogenus $\delta$. 
\end{thm}
\noindent
This is the same as our Theorem~\ref{FMrecast}.



\bibliographystyle{nyjalpha}
\ifx\undefined\bysame
\newcommand{\bysame}{\leavevmode\hbox to3em{\hrulefill}\,}
\fi

\end{document}